\documentclass{amsart}
\usepackage{amssymb}
\usepackage{amsmath, amssymb, euscript,  enumerate}
\usepackage{esint}

\newtheorem{theorem}{Theorem}
\newtheorem{proposition}{Proposition}
\newtheorem{lemma}{Lemma}
\newtheorem{corollary}{Corollary}
\newtheorem{remark}{Remark}
\numberwithin{equation}{section}

\newcommand{\pdt}[0]{\frac{\partial}{\partial t}}
\newcommand{\Rc}[0]{\mathrm{Ric}}
\newcommand{\Rm}[0]{\mathrm{Rm}}

\begin{document}
\title{A local curvature estimate for the Ricci flow}

\author{Brett Kotschwar}
\email{kotschwar@asu.edu}
\address{School of Mathematical and Statistical Sciences,
	Arizona State University, Tempe, AZ 85287, USA}
\author{Ovidiu Munteanu}
\email{ovidiu.munteanu@uconn.edu}
\address{Department of Mathematics, University of Connecticut, Storrs, CT
	06268, USA}
\author{Jiaping Wang}
\email{jiaping@math.umn.edu}
\address{School of Mathematics, University of Minnesota, Minneapolis, MN
	55455, USA}

\thanks{The first author was partially supported by Simons Foundation grant \#359335. The second author was partially supported by NSF grant DMS-1506220.}

\date{}

\begin{abstract}
We show that the norm of the Riemann curvature tensor
of any smooth solution to the Ricci flow can be explicitly estimated
in terms of its initial values on a given ball, a local uniform bound on the Ricci tensor, and the elapsed time.
This provides a new, direct proof of a result of \v{S}es\v{u}m, which asserts that the curvature of a solution
on a compact manifold
cannot blow up while the Ricci curvature remains bounded, and extends its conclusions to the noncompact setting.
We also prove that the Ricci curvature must blow up at least linearly
along a subsequence at a finite time singularity.
\end{abstract}

\maketitle
\section{Introduction}
Let $M$ be a smooth $n$-dimensional manifold and $g(t)$ a solution to the Ricci flow
\begin{equation}
\label{eq:rf}
\pdt g = -2\Rc(g)
\end{equation}
on $M$ defined on a maximal interval $[0, T)$ with $0 < T \leq \infty$. When $M$ is compact,
Hamilton's long-time existence criterion \cite{Hamilton3D}
asserts that either $T=\infty$ or the the maximum of the norm of the Riemann curvature tensor blows up at $T$, that is, $\lim_{t\nearrow T} \sup_{M}|\Rm|_{g(t)} = \infty$. In other words, provided the sectional
curvatures of a solution
are uniformly bounded on a finite interval $[0, T)$, the solution may be extended to a larger interval $[0, T+\epsilon)$.  

The arrival of a finite-time singularity for a solution on a compact manifold is therefore characterized by the blow-up of $|\Rm|$.  
It is of considerable interest to try to express this criterion in terms of a quantity simpler than the norm of the full curvature tensor.
One of the first improvements in this direction was made by \v{S}es\v{u}m \cite{Sesum}, who proved that, in the above situation, if $T < \infty$,
then $\limsup_{t\nearrow T}\sup_{M}|\Rc|= \infty$. Her result has since been generalized in a number of directions.
It has been conjectured, in fact, that the scalar curvature must also blow up in this case, that is, one must actually have $\limsup_{t\nearrow T}\sup_M R = \infty$.
In dimension three, this is a consequence of the Hamilton-Ivey estimate \cite{HamiltonSingularities, Ivey}, 
and it is true for all K\"ahler solutions by a theorem of Zhang \cite{Zhang}. Recent efforts have made considerable progress toward
resolving this conjecture and clarifying the relationship between scalar curvature and singularity formation.
See, for example, \cite{BamlerZhang}, \cite{CaoTran}, \cite{EndersMuellerTopping}, \cite{He}, \cite{Knopf}, \cite{LeSesum},
\cite{Simon1, Simon2}, \cite{Wang1, Wang2}, and the references therein.

On noncompact $M$, the loss of a uniform curvature bound is no longer necessarily coincident with the arrival of a singularity.
There are by now many constructions of smooth, complete solutions whose curvature is unbounded on every time-slice,
and Giesen and Topping \cite{GiesenTopping} and Cabezas-Rivas and Wilking \cite{CabezasRivasWilking} 
have separately constructed examples of solutions $g(t)$ which are defined smoothly for $t\in [0, T)$ and
possess a uniform curvature bound on $[0, T_0]$, but which have unbounded curvature on $[T_1, T)$ for some $T_0 < T_1$.
It is possible, even, that there exist solutions for which $\sup_M|\Rm|_{g(0)}(x, 0) < \infty$ but $\sup_{M\times[0, \epsilon)}|\Rm|_{g(t)} = \infty$ for \emph{all} $\epsilon > 0$, although there are currently
no complete examples known. A recent result of Topping \cite{ToppingUniqueness} implies that they do not occur in two dimensions.
The preservation of a uniform curvature bound
is closely related to the uniqueness of Shi's solutions
\cite{Shi}, which, in general dimensions, is only presently known to hold within the class of complete 
solutions of bounded curvature (\cite{ChenZhu}; cf. \cite{Kotschwar}, \cite{ToppingUniqueness}). 

Hamilton's long-time existence criterion nevertheless has a partial analog for noncompact $M$: it is still true
that if 
$g(t)$ is a complete solution on $M\times[0, T)$ with $0 < T < \infty$ and $\sup_{M\times [0, T)}|\Rm|_{g(t)} < \infty$, then $g(t)$ can be extended smoothly to a solution
on $[0, T+\epsilon)$. Here, through the fundamental estimates of Bando \cite{Bando} and Shi \cite{Shi},
one can parlay the uniform curvature bound into instantaneous uniform bounds on all derivatives of the curvature.
With these, one can show that $g(t)$ converges smoothly to a complete limit metric
$g(T)$ of bounded curvature, and, from there, use the short-time existence result of Shi \cite{Shi} to restart the flow.

As before, it is desirable to express this criterion in terms of a simpler object than the full curvature tensor. 
The first question to ask is whether an analog of \v{S}es\v{u}m's theorem is valid. As we have noted, we must now contend with the new possibility
that the Riemann curvature tensor may become unbounded instantaneously. This is an obstacle to directly adapting the argument-by-contradiction in \cite{Sesum},
which relies on the extraction of a limit of a sequence of solutions rescaled by factors comparable to the spatial 
maximum of curvature. (See, however, \cite{MaCheng} for an approach along these lines.)
Recently, it was proven in \cite{Kotschwar} that, if the curvature of a smooth complete solution is initially bounded, and
if the Ricci curvature is uniformly bounded along the flow, then the Riemann curvature tensor must also
remain uniformly bounded for a short time. However, this result is obtained indirectly, as a consequence of a uniqueness theorem,
and the length of the interval on which the curvature of the solution is guaranteed to remain bounded is nonexplicit.

The statements of the short-time curvature bound in \cite{Kotschwar} and the long-time existence criterion in \cite{Sesum}
raise the question of whether it is possible to simply estimate the growth of the curvature tensor locally and explicitly in terms of the Ricci curvature,
and thereby simultaneously obtain effective proofs of both of these results. Such estimates have some precedent for the Ricci flow: in \cite{MW}, it is shown
that, on a gradient shrinking soliton, a bound on $\Rc$ implies a polynomial
growth bound on $\Rm$. In the present paper, we show that the integral estimates in \cite{MW} which are the basis of these bounds
can be adapted to general smooth solutions to the Ricci flow. Our main result is the following local estimate.

\begin{theorem}
\label{Rm}Let $\left( M^n,g\left( t\right) \right) $ be a smooth solution to the
Ricci flow defined for $0\leq t\leq T$. Suppose that there exist
constants $A$, $K>0$ and a point $x_0\in M$ such that the ball $B_{g(0)}(x_0, A/\sqrt{K})$ is compactly contained in $M$
and
\begin{equation*}
\left\vert \mathrm{Ric}\right\vert\leq K\quad\mbox{ on }\
B_{g\left( 0\right) }\left( x_{0},\frac{A}{\sqrt{K}}\right)\times [0, T] .
\end{equation*}%
Then there are constants $c$, $\alpha$, $\beta >0$, depending only on the dimension $n$, such that 
\begin{equation*}
\left\vert \mathrm{Rm}\right\vert\left( x_{0}, T\right)
\leq ce^{c\left( KT+A\right) }\left( 1+\left( \frac{\Lambda _{0}}{K}\right) ^{\alpha }+\left( \frac{1}{KT}+A^{-2}\right)
^{\beta }\right) \left(
K\left( 1+A^{-2}\right) +\Lambda _{0}\right) ,
\end{equation*}%
where 
\begin{equation*}
\Lambda _{0}:=\sup_{B_{g\left( 0\right) }\left( x_{0},\frac{A}{\sqrt{K}}%
\right) }\left\vert \mathrm{Rm}\right\vert(x, 0).
\end{equation*}
\end{theorem}

This estimate gives a new, direct, proof of \v{S}es\v{u}m's theorem in the compact case,
and extends its conclusion to solutions on noncompact manifolds.
\begin{corollary}\label{cor:rcbound}
 Suppose that $(M^n, g(t))$ is a smooth solution to the Ricci flow defined for $0\leq t < T < \infty$,
 satisfying
 \[
 \quad K := \sup_{M\times [0, T)}|\Rc|(x, t) < \infty. 
\]
Then
\begin{enumerate}
 \item[(a)] There is a smooth metric $\bar{g}$ on $M$, uniformly equivalent to $g(0)$, 
 to which $g(t)$ converges, locally smoothly, as $t\nearrow T$.
 
 \item[(b)] If, in addition, $(M, g(0))$ is complete and
 \[
  \Lambda := \sup_M|\Rm|(x, 0) < \infty,
 \]
then
\[
\sup_{M\times[0, T)}|\Rm|(x, t) < \infty,
\]
 and $g(t)$ extends smoothly to a complete solution
on $[0, T+\epsilon)$ for some $\epsilon > 0$ depending on $n$, $K$, $T$, and $\Lambda$.
\end{enumerate}
\end{corollary}
Part (a) may be proven by an argument analogous to that for compact solutions, by considering the convergence on compact sets.
See Section 14 of \cite{Hamilton3D} and Section 6.7 of \cite{CK}. This demonstrates that, on an \emph{arbitrary} smooth solution,
the curvature tensor cannot blow up at a fixed point so long as the Ricci tensor remains bounded in a surrounding neighborhood.
The uniform bound in (b) follows directly from 
the application of Theorem \ref{Rm} to balls of fixed radius.  It asserts, in particular, that the curvature tensor
cannot instantaneously become unbounded in space so long as the solution continues to move with bounded speed.
This bound is also Theorem 1.4 in the paper \cite{MaCheng}.
The proof there (which is based on a blow-up argument), however, makes use of an
 implicit assumption that the curvature tensor remains bounded for a short time \cite{Cheng}.
 
There are a variety of other applications which one may obtain from the implication that an initial bound on $|\Rm|$
and a uniform bound on $|\Rc|$ guarantees a uniform bound on $|\Rm|$. For example, it follows from Theorem 18.2 in \cite{HamiltonSingularities},
that if a complete solution $(M, g(t))$ has the property that $|\Rm|(x, 0)\to 0$ as $x\to\infty$, then it continues to do so as long as $|\mathrm{Ric}|$ remains
bounded. Our estimates also immediately imply an improvement of the dependencies of the constants in Shi's derivative estimates \cite{Shi}.
\begin{corollary}\label{cor:bbs}
  Let $(M, g(t))$ be a smooth solution to Ricci flow defined for $0\leq t \leq T$.  Suppose that 
  $B_{g(0)}(x_0, A)$ is compactly contained in $M$ for some $A > 0$ and $x_0 \in M$ and let
  \[
      \Lambda := \sup_{B_{g(0)}(x_0, A)}|\Rm|(x, 0), \quad K := \sup_{B_{g(0)}(x_0, A)\times [0, T]}|\Rc|(x, t).
  \]
Then, for all $m = 0, 1, 2, \ldots$,
\[
    \sup_{B_{g(0)}(x_0, \frac{A}{2})\times[0, T]}t^{m/2}|\nabla^{(m)}\Rm|(x, t)\leq C(m, n, A, K, T,  \Lambda).
\]
\end{corollary}

Our proof of Theorem \ref{Rm}  is based on the following integral estimate, which may itself be of some independent interest.

\begin{proposition}
\label{Int}Let $\left( M^n,g\left( t\right) \right)$ be a smooth solution to the
Ricci flow defined for $0\leq t \leq T$. Assume that there exist
constants $A$, $K>0$ and a point $x_0\in M$ such that the ball $B_{g(0)}(x_0, A/\sqrt{K})$ is compactly contained in $M$
and that
\begin{equation*}
\left\vert \mathrm{Ric}\right\vert \leq K\quad\mbox{ on }\
B_{g\left( 0\right) }\left( x_{0},\frac{A}{\sqrt{K}}\right)\times [0,T] .
\end{equation*}%
Then, for any $p\geq 3,$ there exists $c=c\left(n, p\right) >0$ so that for all $0\leq t\leq T$
\begin{gather*}
\int_{B_{g\left( 0\right) }\left( x_{0},\frac{1}{2}\frac{A}{\sqrt{K}}\right) }\left\vert \mathrm{Rm}\right\vert^p(x, t)\leq ce^{cKT}\int_{B_{g\left( 0\right) }\left( x_{0},\frac{A}{\sqrt{K}}\right) }\left\vert \mathrm{Rm}\right\vert^{p}(x, 0) \\
+cK^{p}\left( 1+A^{-2p}\right) e^{cKT}\ \mathrm{Vol}_{g\left( t\right)
}\left( B_{g\left( 0\right) }\left( x_{0},\frac{A}{\sqrt{K}}\right)
\right).
\end{gather*}
\end{proposition}

The proof of this proposition is an adaptation of the method in \cite{MW}.  In that reference, the reduction from $\Rm$ to $\Rc$ is made possible
by the soliton identities. Here, it is made possible by the fact that the evolution of $\Rm$ can be expressed purely in terms of the second covariant
derivatives of $\Rc$, see Lemma \ref{L} below. An analog of the above inequality is also valid for $p=2$, see Remark \ref{rem:p2} below.  We note that
Xu \cite{Xu} (cf. \cite{Yang}) has also obtained local curvature estimates for the Ricci flow  by related integral methods.

We prove Proposition \ref{Int} in the following section by the means of an energy estimate.  In Section \ref{sec:rmproof},
we combine it with a standard iteration argument  to
prove Theorem \ref{Rm}.

In Section \ref{sec:blowup}, we use a variation on the above methods to investigate the possible rates of blow-up
of the Ricci curvature at a finite-time singularity. Note that Part (b) of Corollary \ref{cor:rcbound} implies that, if a complete solution $g(t)$ to the Ricci flow
is defined on a maximal interval $[0, T)$ with $T < \infty$ and $\sup_M|\Rm|(x, t) < \infty$ for all $0\leq t < T$,
then the Ricci curvature must blow-up as $t\nearrow T$, i.e., there
must exist a sequence of times $t_i \nearrow T$ along which $\lim_{t_i\nearrow T}\sup_M|\Rc|(x, t_i) = \infty$.
The theorem does not say anything, however, about how rapidly the blow-up must occur.  By contrast,
for the Riemann curvature tensor, it is known not only that one actually has $\lim_{t\nearrow T}\sup_M|\Rm|(x, t) = \infty$
at $T$, but that, in fact, $\sup_M |\Rm|(x, t) \geq 1/(8(T-t))$. This is a simple consequence of the parabolic maximum principle;
see, e.g., Lemma 8.7 of \cite{ChowLuNi}. It is natural to ask whether there is a corresponding minimal rate of blow-up for the Ricci curvature.
Our methods here allow us to give at least a partial answer to this question. In Theorem \ref{Blowup} of Section \ref{sec:blowup},
using a somewhat more careful iteration argument,
we prove that in the situation just described, we must at least have
\[
      \sup_M|\Rc|(x, t_i) \geq \frac{\varepsilon}{T-t_i}
\]
along a subsequence $t_i \nearrow T$ for some $\varepsilon > 0$ depending only on the dimension.

\section{The main estimate}
 
The following differential inequality is the primary ingredient in the proof of Proposition \ref{Int}. Let $K(t)$ be a positive $C^1$-function on $[0,T]$.

\begin{proposition}
\label{A}
Let $g(t)$ be a smooth solution to the Ricci flow on $M^n$, defined for $0 \leq t\leq T$, satisfying the uniform Ricci bound
\[
 \sup_{\Omega} \left\vert\mathrm{Ric}\right\vert(x, t)\leq K(t)
\]
on some open $\Omega \subset M$, for all $0\leq t\leq T$. Then, for any $p \geq 3$, there are constants $c_1$, $c_2 > 0$ depending only on $n$ and $p$,
such that 
\begin{align}
\begin{split}\label{eq:diffeq}
&\frac{d}{dt}\left( \int_{M}\left\vert \mathrm{Rm}\right\vert ^{p}\phi
^{2p}+\frac{1}{K}\int_{M}\left\vert \mathrm{Ric}\right\vert ^{2}\left\vert \mathrm{Rm}%
\right\vert ^{p-1}\phi ^{2p}+c_{1}K\int_{M}\left\vert \mathrm{Rm}\right\vert
^{p-1}\phi ^{2p}\right)   \\
&\qquad\qquad\leq c_2K\int_{M}\left\vert \mathrm{Rm}\right\vert ^{p}\phi^{2p}
+c_2K\int_M \left\vert \mathrm{Rm}\right\vert ^{p-1}\left\vert \nabla \phi\right\vert^2
\phi^{2p-2}\\
&\qquad\qquad-\frac{K'}{K^2}\int_{M}\left\vert \mathrm{Ric}\right\vert ^{2}\left\vert \mathrm{Rm}
\right\vert ^{p-1}\phi ^{2p}+c_{1}K' \int_{M}\left\vert \mathrm{Rm}
\right\vert ^{p-1}\phi ^{2p},
\end{split}
\end{align}
for any Lipschitz function $\phi (x)$ with support in $\Omega$.
\end{proposition}	
\subsection{Proof of Proposition \ref{A}}
The proof requires a bit of computation. We first recall the evolution equations for the curvature see, e.g, Chapter 6 of \cite{CK}.

\begin{lemma}
\label{L} There exists a constant $c > 0$ such that
\begin{align}
\label{eq:rcnormev}
\left\vert \nabla \mathrm{Ric}\right\vert ^{2} &\leq \frac{1}{2}\left(
\Delta -\partial _{t}\right) \left\vert \mathrm{Ric}\right\vert
^{2}+cK^2\left\vert \mathrm{Rm}\right\vert, \\
\label{eq:rmnormev}
\left\vert \nabla \mathrm{Rm}\right\vert ^{2} &\leq\frac{1}{2}\left(
\Delta -\partial _{t}\right) \left\vert \mathrm{Rm}\right\vert
^{2}+c\left\vert \mathrm{Rm}\right\vert ^{3}, \\
\begin{split}\label{eq:rmev}
\partial _{t}R_{ijk}^{l} &=g^{lq}\left(\nabla _{i}\nabla _{q}R_{jk}+\nabla
_{j}\nabla _{i}R_{kq}+\nabla _{j}\nabla _{k}R_{iq}\right)\\
&\phantom{=}
-g^{lq}\left(\nabla _{i}\nabla _{j}R_{kq}+\nabla _{i}\nabla
_{k}R_{jq}+\nabla _{j}\nabla _{q}R_{ik}\right),
\end{split}
\end{align}
on $\Omega\times[0, T]$.
\end{lemma}

In the argument below, we will use $c$ to represent a positive constant depending only on $n$
and $p$, and will simply write $\vert\cdot\vert$ for the various tensor norms $\vert\cdot\vert_{g(t)}$
induced by $g(t)$. All integrals are taken relative to the evolving Riemannian measure $dv_t = dv_{g(t)}$,
and we use the standard convention that $A\ast B$ represents some linear combination of contractions
of the tensor product $A\otimes B$ relative to the metric $g(t)$. As above, we use $R$ to denote the scalar curvature of $g(t)$.

To begin, we use equation \eqref{eq:rmev} to write
\begin{align*}
	\partial_t\left\vert \mathrm{Rm}\right\vert ^{2} &=\partial_t\left(
	g^{ia}g^{jb}g^{kc}g_{ld}R_{ijk}^{l}R_{abc}^d\right)\\
	&=\nabla ^{2}\mathrm{Ric}\ast \mathrm{Rm}+\mathrm{Ric}\ast \mathrm{Rm}\ast 
	\mathrm{Rm}.
\end{align*}%
Since the volume form $dv_t $ evolves by $\partial_t dv_t =-R\,dv_t$, we then have 
\begin{align*}
\frac{d}{dt}\int_{M}\left\vert \mathrm{Rm}\right\vert ^{p}\phi ^{2p}
&=\int_{M}\frac{\partial}{\partial t}\left( \left\vert \mathrm{Rm}\right\vert ^{p}\right)
\phi ^{2p}-\int_{M}R\left\vert \mathrm{Rm}\right\vert ^{p}\phi ^{2p} \\
&\leq c\int_{M}\left\vert \mathrm{Rm}\right\vert ^{p-2}\left( \nabla ^{2}%
\mathrm{Ric}\ast \mathrm{Rm}\right) \phi ^{2p}+cK\int_{M}\left\vert \mathrm{Rm%
}\right\vert ^{p}\phi ^{2p}
\end{align*}%
for some $c > 0$.
Integrating by parts, we find that%
\begin{align*}
\int_{M}\left\vert \mathrm{Rm}\right\vert ^{p-2}\left( \nabla ^{2}\mathrm{Ric%
}\ast \mathrm{Rm}\right) \phi ^{2p} &=\int_{M}\left\vert \mathrm{Rm}%
\right\vert ^{p-2}\left( \nabla \mathrm{Ric}\ast \nabla \mathrm{Rm}\right)
\phi ^{2p} \\
&\phantom{=}+\int_{M}\left( \nabla \mathrm{Ric}\ast \nabla \left\vert \mathrm{Rm}%
\right\vert ^{p-2}\ast \mathrm{Rm}\right) \phi ^{2p} \\
&\phantom{=}+\int_{M}\left\vert \mathrm{Rm}\right\vert ^{p-2}\left( \nabla \mathrm{Ric}%
\ast \nabla \phi ^{2p}\ast \mathrm{Rm}\right),
\end{align*}%
from which we readily obtain that 
\begin{align}
\begin{split}
\frac{d}{dt}\int_{M}\left\vert \mathrm{Rm}\right\vert ^{p}\phi ^{2p} &\leq
c\int_{M}\left\vert \nabla \mathrm{Ric}\right\vert \left\vert \nabla 
\mathrm{Rm}\right\vert \left\vert \mathrm{Rm}\right\vert ^{p-2}\phi ^{2p}
\label{b3} \\
&\phantom{\leq}+c\int_{M}\left\vert \nabla \mathrm{Ric}\right\vert \left\vert \mathrm{Rm}%
\right\vert ^{p-1}\left\vert \nabla \phi \right\vert \phi ^{2p-1}  \\
&\phantom{\leq}+cK\int_{M}\left\vert \mathrm{Rm}\right\vert ^{p}\phi ^{2p}.
\end{split}
\end{align}%
We may estimate the first term on the right of \eqref{b3} by
\begin{align}
\begin{split}
c\int_{M}\left\vert \nabla \mathrm{Ric}\right\vert \left\vert \nabla \mathrm{%
Rm}\right\vert \left\vert \mathrm{Rm}\right\vert ^{p-2}\phi ^{2p} 
&\leq %
\frac{1}{2K}\int_{M}\left\vert \nabla \mathrm{Ric}\right\vert ^{2}\left\vert 
\mathrm{Rm}\right\vert ^{p-1}\phi ^{2p}  \label{b4} \\
&\phantom{\leq}+cK\int_{M}\left\vert \nabla \mathrm{Rm}\right\vert ^{2}\left\vert \mathrm{%
Rm}\right\vert ^{p-3}\phi ^{2p},
\end{split}
\end{align}%
and the second term, similarly, by
\begin{align}\label{b5}
\begin{split} 
&c\int_{M}\left\vert \nabla \mathrm{Ric}\right\vert \left\vert \mathrm{Rm}%
\right\vert ^{p-1}\left\vert \nabla \phi \right\vert \phi ^{2p-1}\\ 
 &\qquad\leq \frac{1}{2K}\int_{M}\left\vert \nabla \mathrm{Ric}\right\vert^{2}\left\vert \mathrm{Rm}\right\vert ^{p-1}\phi ^{2p}
+cK\int_{M}\left\vert \mathrm{Rm}\right\vert ^{p-1} \left\vert\nabla\phi\right\vert^2
\phi ^{2p-2}.
\end{split}
\end{align}
Using (\ref{b4}) and (\ref{b5}) in (\ref{b3}), it follows that 
\begin{align}
\begin{split}
\frac{d}{dt}\int_{M}\left\vert \mathrm{Rm}\right\vert ^{p}\phi ^{2p} &\leq
\frac{1}{K}\int_{M}\left\vert \nabla \mathrm{Ric}\right\vert ^{2}\left\vert \mathrm{Rm}%
\right\vert ^{p-1}\phi ^{2p}  \label{b7} \\
&\phantom{\leq}+cK\int_{M}\left\vert \nabla \mathrm{Rm}\right\vert ^{2}\left\vert \mathrm{%
Rm}\right\vert ^{p-3}\phi ^{2p} \\
&\phantom{\leq}+cK\int_{M}\left\vert \mathrm{Rm}\right\vert ^{p-1} \left\vert\nabla\phi\right\vert^2
\phi ^{2p-2}\\
&\phantom{\leq}+cK\int_{M}\left\vert \mathrm{Rm}\right\vert ^{p}\phi ^{2p}.  \\
\end{split}
\end{align}

We now set about to estimate the first two terms on the right of \eqref{b7}.
Using equation \eqref{eq:rcnormev} from Lemma \ref{L}, we get 
\begin{align*}
\frac{1}{K}\int_{M}\left\vert \nabla \mathrm{Ric}\right\vert ^{2}\left\vert \mathrm{Rm}%
\right\vert ^{p-1}\phi ^{2p}&\leq \frac{1}{2K}\int_{M}\left( \left( \Delta
-\partial _{t}\right) \left\vert \mathrm{Ric}\right\vert ^{2}\right)
\left\vert \mathrm{Rm}\right\vert ^{p-1}\phi ^{2p}\\
&\phantom{\leq}+cK\int_{M}\left\vert 
\mathrm{Rm}\right\vert ^{p}\phi ^{2p}.
\end{align*}%
Hence, integrating by parts, it follows that%
\begin{align}
\begin{split}\label{b8}
\frac{1}{K}\int_{M}\left\vert \nabla \mathrm{Ric}\right\vert ^{2}\left\vert \mathrm{Rm}%
\right\vert ^{p-1}\phi ^{2p} 
&\leq -\frac{1}{2K}\int_{M}\left\langle \nabla
\left\vert \mathrm{Ric}\right\vert ^{2},\nabla \left\vert \mathrm{Rm}%
\right\vert ^{p-1}\right\rangle \phi ^{2p}   \\
&\phantom{\leq}-\frac{1}{2K}\int_{M}\left\langle \nabla \left\vert \mathrm{Ric}\right\vert
^{2},\nabla \phi ^{2p}\right\rangle \left\vert \mathrm{Rm}\right\vert ^{p-1}
 \\
&\phantom{\leq}-\frac{1}{2K}\frac{d}{dt}\int_{M}\left\vert \mathrm{Ric}\right\vert
^{2}\left\vert \mathrm{Rm}\right\vert ^{p-1}\phi ^{2p}   \\
&\phantom{\leq}+\frac{1}{2K}\int_{M}\left\vert \mathrm{Ric}\right\vert ^{2}\partial
_{t}\left\vert \mathrm{Rm}\right\vert ^{p-1}\phi ^{2p}   \\
&\phantom{\leq}+cK\int_{M}\left\vert \mathrm{Rm}\right\vert ^{p}\phi ^{2p}. 
\end{split}
\end{align}

Since $\left\vert \mathrm{Ric}\right\vert(t) \leq K(t)$ on the support of $\phi$, for the first term in \eqref{b8}, we have
\begin{align}
\begin{split}\label{b8b}
-\frac{1}{2K}\int_{M}\left\langle \nabla \left\vert \mathrm{Ric}\right\vert
^{2},\nabla \left\vert \mathrm{Rm}\right\vert ^{p-1}\right\rangle \phi ^{2p}
&\leq c\int_{M}\left\vert \nabla \mathrm{Ric}\right\vert \left\vert \nabla 
\mathrm{Rm}\right\vert \left\vert \mathrm{Rm}\right\vert ^{p-2}\phi ^{2p} \\
&\leq \frac{1}{10K}\int_{M}\left\vert \nabla \mathrm{Ric}\right\vert
^{2}\left\vert \mathrm{Rm}\right\vert ^{p-1}\phi ^{2p} \\
&\phantom{\leq}+cK\int_{M}\left\vert \nabla \mathrm{Rm}\right\vert ^{2}\left\vert \mathrm{%
Rm}\right\vert ^{p-3}\phi ^{2p}.
\end{split}
\end{align}%
Furthermore, as in (\ref{b5}), we have
\begin{align}
\begin{split}
&-\frac{1}{2K}\int_{M}\left\langle \nabla \left\vert \mathrm{Ric}\right\vert
^{2},\nabla \phi ^{2p}\right\rangle \left\vert \mathrm{Rm}\right\vert ^{p-1}
\label{b9}\\
&\qquad\qquad\leq c\int_{M}\left\vert \nabla \mathrm{Ric}\right\vert \left\vert \mathrm{%
Rm}\right\vert ^{p-1}\left\vert \nabla \phi \right\vert \phi ^{2p-1}  \\
&\qquad\qquad\leq \frac{1}{10K}\int_{M}\left\vert \nabla \mathrm{Ric}\right\vert
^{2}\left\vert \mathrm{Rm}\right\vert ^{p-1}\phi ^{2p} +cK\int_{M}\left\vert \mathrm{Rm}\right\vert ^{p-1} \left\vert\nabla\phi\right\vert^2
\phi ^{2p-2}  
\end{split}
\end{align}
for the second term in \eqref{b8}.
Hence, using \eqref{b8b} and \eqref{b9}, equation (\ref{b8}) implies 
\begin{align}
\begin{split}\label{b10}
&\frac{4}{5K}\int_{M}\left\vert \nabla \mathrm{Ric}\right\vert
^{2}\left\vert \mathrm{Rm}\right\vert ^{p-1}\phi ^{2p}   \\
&\qquad\leq cK\int_{M}\left\vert \nabla \mathrm{Rm}\right\vert ^{2}\left\vert 
\mathrm{Rm}\right\vert ^{p-3}\phi ^{2p}-\frac{1}{2K}\frac{d}{dt}%
\int_{M}\left\vert \mathrm{Ric}\right\vert ^{2}\left\vert \mathrm{Rm}%
\right\vert ^{p-1}\phi ^{2p}  \\
&\qquad\phantom{\leq}+\frac{1}{2K}\int_{M}\left\vert \mathrm{Ric}\right\vert ^{2}\partial
_{t}\left\vert \mathrm{Rm}\right\vert ^{p-1}\phi ^{2p}+cK\int_{M}\left\vert 
\mathrm{Rm}\right\vert ^{p}\phi ^{2p}  \\
&\qquad\phantom{\leq}+cK\int_{M}\left\vert \mathrm{Rm}\right\vert ^{p-1} \left\vert\nabla\phi\right\vert^2
\phi ^{2p-2}. 
\end{split}
\end{align}

Applying again equation \eqref{eq:rmev} of Lemma \ref{L}, we may estimate the third term on the right of \eqref{b10} by
\begin{align}
\begin{split}
\frac{1}{2K}\int_{M}\left\vert \mathrm{Ric}\right\vert ^{2}\partial
_{t}\left\vert \mathrm{Rm}\right\vert ^{p-1}\phi ^{2p} 
&\leq
\frac{c}{K}\int_{M}\left\vert \mathrm{Ric}\right\vert ^{2}\left( \nabla ^{2}\mathrm{%
Ric}\ast \mathrm{Rm}\right) \left\vert \mathrm{Rm}\right\vert ^{p-3}\phi
^{2p} \label{f1}\\
&\phantom{\leq} +cK\int_{M}\left\vert \mathrm{Rm}\right\vert ^{p}\phi ^{2p}
\end{split}
\end{align}
where we have used that $|\Rc|\leq c|\Rm|$.

Integrating by parts on the first term of the right hand side of (\ref{f1}),
it follows that  
\begin{align*}
	\int_{M}\left\vert \mathrm{Ric}\right\vert ^{2}\left( \nabla ^{2}\mathrm{Ric}%
	\ast \mathrm{Rm}\right) \left\vert \mathrm{Rm}\right\vert ^{p-3}\phi ^{2p}
	&=\int_{M}\left( \nabla \mathrm{Ric}\ast \nabla \left\vert \mathrm{Ric}%
	\right\vert ^{2}\ast \mathrm{Rm}\right) \left\vert \mathrm{Rm}\right\vert
	^{p-3}\phi ^{2p} \\
	&\phantom{=}+\int_{M}\left( \nabla \mathrm{Ric}\ast \nabla \mathrm{Rm}\right)
	\left\vert \mathrm{Ric}\right\vert ^{2}\left\vert \mathrm{Rm}\right\vert
	^{p-3}\phi ^{2p} \\
	&\phantom{=}+\int_{M}\left( \nabla \mathrm{Ric}\ast \nabla \left\vert \mathrm{Rm}%
	\right\vert ^{p-3}\ast \mathrm{Rm}\right) \left\vert \mathrm{Ric}\right\vert
	^{2}\phi ^{2p} \\
	&\phantom{=}+\int_{M}\left( \nabla \mathrm{Ric}\ast \nabla \phi ^{2p}\ast \mathrm{Rm}%
	\right) \left\vert \mathrm{Ric}\right\vert ^{2}\left\vert \mathrm{Rm}%
	\right\vert ^{p-3}.
\end{align*}
Hence, using again that $\left\vert \mathrm{Ric}\right\vert \leq c\left\vert 
\mathrm{Rm}\right\vert $ and also that $\left\vert \mathrm{Ric}\right\vert
\leq K$, we can estimate the above by 
\begin{align*}
\frac{c}{K}\int_{M}\left\vert \mathrm{Ric}\right\vert ^{2}\left( \nabla ^{2}\mathrm{Ric}%
\ast \mathrm{Rm}\right) \left\vert \mathrm{Rm}\right\vert ^{p-3}\phi ^{2p}
&\leq c\int_{M} \left\vert \nabla \mathrm{Ric}\right\vert ^{2}\left\vert \mathrm{Rm}\right\vert ^{p-2}\phi ^{2p} \\
& \phantom{\leq}+c\int_{M}\left\vert \nabla \mathrm{Ric}\right\vert \left\vert \nabla 
\mathrm{Rm}\right\vert \left\vert \mathrm{Rm}\right\vert ^{p-2}\phi ^{2p} \\
& \phantom{\leq}+c\int_{M}\left\vert \nabla \mathrm{Ric}\right\vert \left\vert \nabla \phi
\right\vert \left\vert \mathrm{Rm}\right\vert ^{p-1}\phi ^{2p-1}\\
& \phantom{\leq}\leq c\int_{M}\left\vert \nabla \mathrm{Ric}\right\vert \left\vert \nabla 
\mathrm{Rm}\right\vert \left\vert \mathrm{Rm}\right\vert ^{p-2}\phi ^{2p} \\
& \phantom{\leq}+c\int_{M}\left\vert \nabla \mathrm{Ric}\right\vert \left\vert \nabla \phi
\right\vert \left\vert \mathrm{Rm}\right\vert ^{p-1}\phi ^{2p-1}.
\end{align*}
Using this estimate in (\ref{f1}) it follows that 
\begin{align}
\begin{split}
\frac{1}{2K}\int_{M}\left\vert \mathrm{Ric}\right\vert ^{2}\partial
_{t}\left\vert \mathrm{Rm}\right\vert ^{p-1}\phi ^{2p}& \leq
c\int_{M}\left\vert \nabla \mathrm{Ric}\right\vert \left\vert \nabla \mathrm{%
	Rm}\right\vert \left\vert \mathrm{Rm}\right\vert ^{p-2}\phi ^{2p} \label{b11}\\
& \phantom{\leq}+c\int_{M}\left\vert \nabla \mathrm{Ric}\right\vert
\left\vert \nabla \phi \right\vert \left\vert \mathrm{Rm}\right\vert
^{p-1}\phi ^{2p-1} \\
& \phantom{\leq}+cK\int_{M}\left\vert \mathrm{Rm}\right\vert ^{p}\phi ^{2p}.
\end{split}%
\end{align}%
Using the inequality 
\begin{align*}
c\int_{M}\left\vert \nabla \mathrm{Ric}\right\vert \left\vert \nabla 
\mathrm{Rm}\right\vert \left\vert \mathrm{Rm}\right\vert ^{p-2}\phi ^{2p}
&\leq\frac{1}{10K}\int_{M}\left\vert \nabla \mathrm{Ric}\right\vert
^{2}\left\vert \mathrm{Rm}\right\vert ^{p-1}\phi ^{2p} \\
&\phantom{\leq}+cK\int_{M}\left\vert \nabla \mathrm{Rm}\right\vert ^{2}\left\vert \mathrm{%
Rm}\right\vert ^{p-3}\phi ^{2p},
\end{align*}%
together with inequality (\ref{b9}), we then obtain
\begin{align}
\begin{split} \label{b12}
\frac{1}{2K}\int_{M}\left\vert \mathrm{Ric}\right\vert ^{2}\partial
_{t}\left\vert \mathrm{Rm}\right\vert ^{p-1}\phi ^{2p}
&\leq \frac{1}{5K}\int_{M}\left\vert \nabla \mathrm{Ric}\right\vert
^{2}\left\vert \mathrm{Rm}\right\vert ^{p-1}\phi ^{2p}  \\
&\phantom{\leq}+cK\int_{M}\left\vert \nabla \mathrm{Rm}\right\vert ^{2}\left\vert \mathrm{%
Rm}\right\vert ^{p-3}\phi ^{2p}  \\
&\phantom{\leq}+cK\int_{M}\left\vert \mathrm{Rm}\right\vert ^{p}\phi ^{2p}  \\
&\phantom{\leq}+cK\int_{M}\left\vert \mathrm{Rm}\right\vert ^{p-1} \left\vert\nabla\phi\right\vert^2
\phi ^{2p-2}.
\end{split}
\end{align}

Substituting (\ref{b12}) into (\ref{b10}) then gives 
\begin{align}
\begin{split}\label{b13}
&\frac{1}{2K}\int_{M}\left\vert \nabla \mathrm{Ric}\right\vert
^{2}\left\vert \mathrm{Rm}\right\vert ^{p-1}\phi ^{2p}\\  
&\qquad\leq cK\int_{M}\left\vert \nabla \mathrm{Rm}\right\vert ^{2}\left\vert 
\mathrm{Rm}\right\vert ^{p-3}\phi ^{2p}-\frac{1}{2K}\frac{d}{dt}%
\int_{M}\left\vert \mathrm{Ric}\right\vert ^{2}\left\vert \mathrm{Rm}%
\right\vert ^{p-1}\phi ^{2p}  \\
&\qquad\phantom{\leq}+cK\int_{M}\left\vert \mathrm{Rm}\right\vert ^{p}\phi ^{2p} 
+cK\int_{M}\left\vert \mathrm{Rm}\right\vert ^{p-1} \left\vert\nabla\phi\right\vert^2
\phi ^{2p-2}. 
\end{split}
\end{align}

This completes our estimate of the first term on the right of (\ref{b7}).  Updated, inequality \eqref{b7} now reads
\begin{align}
\begin{split}\label{b14}
&\frac{d}{dt}\int_{M}\left\vert \mathrm{Rm}\right\vert ^{p}\phi ^{2p}\\
&\qquad\leq -\frac{1}{K}\frac{d}{dt}\int_{M}\left\vert \mathrm{Ric}\right\vert
^{2}\left\vert \mathrm{Rm}\right\vert ^{p-1}\phi ^{2p}+cK\int_{M}\left\vert
\nabla \mathrm{Rm}\right\vert ^{2}\left\vert \mathrm{Rm}\right\vert
^{p-3}\phi ^{2p} \\
&\qquad\phantom{\leq}+cK\int_{M}\left\vert \mathrm{Rm}\right\vert ^{p-1} \left\vert\nabla\phi\right\vert^2\phi ^{2p-2} +cK\int_{M}\left\vert \mathrm{Rm}\right\vert ^{p}\phi^{2p}. 
\end{split}
\end{align}

It remains only to estimate the second term on the right of \eqref{b14}. This may be done more simply. Using \eqref{eq:rmnormev} of Lemma \ref{L}, we find first that 
\begin{align}
\begin{split}\label{b15}
2\int_{M}\left\vert \nabla \mathrm{Rm}\right\vert ^{2}\left\vert \mathrm{Rm}%
\right\vert ^{p-3}\phi ^{2p} &\leq\int_{M}\left( \left( \Delta -\partial
_{t}\right) \left\vert \mathrm{Rm}\right\vert ^{2}\right) \left\vert \mathrm{%
Rm}\right\vert ^{p-3}\phi ^{2p}   \\
&\phantom{\leq}+c\int_{M}\left\vert \mathrm{Rm}\right\vert ^{p}\phi ^{2p}. 
\end{split}
\end{align}
Integrating by parts, and using that $p\geq 3$, we get that
\begin{align*}
\int_{M}\left( \Delta \left\vert \mathrm{Rm}\right\vert ^{2}\right)
\left\vert \mathrm{Rm}\right\vert ^{p-3}\phi ^{2p} &=-\int_{M}\left\langle
\nabla \left\vert \mathrm{Rm}\right\vert ^{2},\nabla \left\vert \mathrm{Rm}%
\right\vert ^{p-3}\right\rangle \phi ^{2p} \\
&\phantom{=}-\int_{M}\left\langle \nabla \left\vert \mathrm{Rm}\right\vert ^{2},\nabla
\phi ^{2p}\right\rangle \left\vert \mathrm{Rm}\right\vert ^{p-3} \\
&\leq c\int_{M}\left\vert \nabla \mathrm{Rm}\right\vert \left\vert \nabla
\phi \right\vert \left\vert \mathrm{Rm}\right\vert ^{p-2}\phi ^{2p-1} \\
&\leq \int_{M}\left\vert \nabla \mathrm{Rm}\right\vert ^{2}\left\vert 
\mathrm{Rm}\right\vert ^{p-3}\phi ^{2p} \\
&\phantom{\leq}+c\int_{M}\left\vert \nabla \phi \right\vert ^{2}\left\vert \mathrm{Rm}%
\right\vert ^{p-1}\phi ^{2p-2},
\end{align*}%
and, furthermore, that 
\begin{equation*}
\int_{M}\left( -\partial _{t}\left\vert \mathrm{Rm}\right\vert ^{2}\right)
\left\vert \mathrm{Rm}\right\vert ^{p-3}\phi ^{2p}\leq -\frac{2}{p-1}%
\frac{d}{dt}\int_{M}\left\vert \mathrm{Rm}\right\vert ^{p-1}\phi
^{2p}+c\int_{M}\left\vert \mathrm{Rm}\right\vert ^{p}\phi ^{2p}.
\end{equation*}
Therefore, it follows from (\ref{b15}) that 
\begin{align}
\begin{split}\label{eq:2ndterm}
\int_{M}\left\vert \nabla \mathrm{Rm}\right\vert ^{2}\left\vert \mathrm{Rm}%
\right\vert ^{p-3}\phi ^{2p} &\leq-\frac{2}{p-1}\frac{d}{dt}%
\int_{M}\left\vert \mathrm{Rm}\right\vert ^{p-1}\phi ^{2p} \\
&\phantom{\leq}+c\int_{M}\left\vert \mathrm{Rm}\right\vert ^{p}\phi ^{2p} \\
&\phantom{\leq}+c\int_{M}\left\vert \mathrm{Rm}\right\vert ^{p-1} \left\vert\nabla\phi\right\vert^2
\phi ^{2p-2}.
\end{split}
\end{align}%
Inserting \eqref{eq:2ndterm} into \eqref{b14}, we conclude that, for any $p\geq 3$, there exist constants $c_1>0$ and $c_2>0$, depending only on $n$ and $p$, such that
\begin{align*}
\begin{split}
&\frac{d}{dt} \int_{M}\left\vert \mathrm{Rm}\right\vert ^{p}\phi
^{2p}+\frac{1}{K}\frac{d}{dt}\int_{M}\left\vert \mathrm{Ric}\right\vert ^{2}\left\vert \mathrm{Rm}%
\right\vert ^{p-1}\phi ^{2p}+c_{1}K\frac{d}{dt}\int_{M}\left\vert \mathrm{Rm}\right\vert
^{p-1}\phi ^{2p}   \\
&\qquad\qquad\leq c_2K\int_{M}\left\vert \mathrm{Rm}\right\vert ^{p}\phi^{2p}+c_2K\int_{M}\left\vert \mathrm{Rm}\right\vert ^{p-1} \left\vert\nabla\phi\right\vert^2
\phi ^{2p-2}.  
\end{split}
\end{align*}
Since 
\begin{align*}
\begin{split}
\frac{1}{K}\frac{d}{dt}\int_{M}\left\vert \mathrm{Ric}\right\vert ^{2}\left\vert \mathrm{Rm}
\right\vert ^{p-1}\phi ^{2p} &= \frac{d}{dt}\left(\frac{1}{K} \int_{M}\left\vert \mathrm{Ric}\right\vert ^{2}\left\vert \mathrm{Rm}
\right\vert ^{p-1}\phi ^{2p}\right)\\
&\phantom{=}+\frac{K'}{K^2}\int_{M}\left\vert \mathrm{Ric}\right\vert ^{2}\left\vert \mathrm{Rm}
\right\vert ^{p-1}\phi ^{2p}
\end{split}
\end{align*}
and 
\begin{equation*}
K\frac{d}{dt}\int_{M}\left\vert \mathrm{Rm}\right\vert
^{p-1}\phi ^{2p}=\frac{d}{dt}\left(K\int_{M}\left\vert \mathrm{Rm}\right\vert
^{p-1}\phi ^{2p}\right)-K'\int_{M}\left\vert \mathrm{Rm}\right\vert
^{p-1}\phi ^{2p},
\end{equation*}
this completes the proof of Proposition \ref{A}.
\subsection{Proof of Proposition \ref{Int}}
Proposition \ref{Int} now follows easily. Our assumptions are that $g(t)$ is a smooth solution
defined on $M \times [0, T]$, and that there exist constants $A$, $K > 0$ and there exists $x_0\in M$ so that the ball $B_{g(0)}(x_0, A/\sqrt{K})$ is compactly
contained in $M$ and we have
\begin{equation*}
	\left\vert \mathrm{Ric}\right\vert\leq K\text{ on }%
	B_{g\left( 0\right) }\left( x_{0},\frac{A}{\sqrt{K}}\right)\times [0,T] .
\end{equation*}%
We wish to prove that, for any $p\geq 3,$ there exists $c=c\left( n,p\right) >0$ so that 
	\begin{gather*}
	\int_{B_{g\left( 0\right) }\left( x_{0},\frac{1}{2}\frac{A}{\sqrt{K}}\right) }\left\vert \mathrm{Rm}\right\vert^{p}(x, \tau)\leq ce^{cKT}\int_{B_{g\left( 0\right) }\left( x_{0},\frac{A}{\sqrt{K}}\right) }\left\vert \mathrm{Rm}\right\vert^{p}(x, 0) \\
+cK^{p}\left( 1+A^{-2p}\right) e^{cKT}\ \mathrm{Vol}_{g\left( \tau\right)
}\left( B_{g\left( 0\right) }\left( x_{0},\frac{A}{\sqrt{K}}\right)
\right),
\end{gather*}
for all $0\leq \tau\leq T$.
Due to the space-time invariance of the Ricci flow, it suffices to prove the proposition for $K=1$. We will assume below, then, that
\begin{equation}
\left\vert\mathrm{Ric}\right\vert(x, t) \leq 1 \quad\mbox{on}\, B_{g\left(0\right)}\left( x_0,A\right)\times [0, T]. \label{b1'}
\end{equation}
It
follows that 
\begin{equation}\label{eq:unifeq}
e^{-2t}g_{ij}\left(x, 0\right) \leq g_{ij}\left(x, t\right) \leq
e^{2t}g_{ij}\left(x, 0\right),
\end{equation}%
for all $0 \leq t \leq T$. 
Consider the cut-off function 
\begin{equation}
\phi \left( x\right) :=\left(\frac{A-d_{g\left( 0\right) }\left(
	x_{0},x\right) }{A}\right)_+  \label{b0}
\end{equation}%
which is Lipschitz with support $\overline{B_{g\left(0\right)}\left( x_{0},A\right)}$.

 Thus, from (\ref{b1'}) we see that $\left\vert \mathrm{Ric}\right\vert \left(x, t\right)\leq 1$ on the support of $\phi$. Furthermore, from \eqref{eq:unifeq} and \eqref{b0}, we see that 
\begin{equation}
\left\vert \nabla \phi \right\vert _{g\left( t\right) }\leq e^{T}\left\vert
\nabla \phi \right\vert _{g\left( 0\right) }\leq A^{-1}e^{T}.  \label{b2}
\end{equation}
Hence, we may apply Proposition \ref{A} to obtain that 
\begin{align}
\begin{split}\label{b16}
&\frac{d}{dt}\left( \int_{M}\left\vert \mathrm{Rm}\right\vert ^{p}\phi
^{2p}+\int_{M}\left\vert \mathrm{Ric}\right\vert ^{2}\left\vert \mathrm{Rm}%
\right\vert ^{p-1}\phi ^{2p}+c_{1}\int_{M}\left\vert \mathrm{Rm}\right\vert
^{p-1}\phi ^{2p}\right)   \\
&\qquad\qquad\leq c_2\int_{M}\left\vert \mathrm{Rm}\right\vert ^{p}\phi^{2p}+c_2\int_{M}\left\vert \mathrm{Rm}\right\vert ^{p-1} \left\vert\nabla\phi\right\vert^2
\phi ^{2p-2}.  
\end{split}
\end{align}
We now estimate the rightmost term in \eqref{b16} using (\ref{b2}) and Young's inequality, obtaining 
\begin{align}\label{b5'}
\begin{split}
\int_{M}\left\vert \mathrm{Rm}\right\vert ^{p-1}\left\vert\nabla\phi\right\vert^2
\phi ^{2p-2}&\leq A^{-2}e^{2T}\int_{M}\left\vert \mathrm{Rm}\right\vert ^{p-1}\phi ^{2p-2}\\
&\leq \int_{M}\left\vert \mathrm{Rm}\right\vert ^{p}\phi ^{2p}  \\
&\phantom{\leq}+cA^{-2p}e^{2Tp}\mathrm{Vol}_{g\left( t\right) }\left( B_{g\left( 0\right)
}\left( x_{0},A\right) \right).
\end{split}
\end{align}%

Consider the function 
\begin{equation*}
U\left( t\right) :=\int_{M}\left\vert \mathrm{Rm}\right\vert ^{p}\phi
^{2p}+\int_{M}\left\vert \mathrm{Ric}\right\vert ^{2}\left\vert \mathrm{Rm}%
\right\vert ^{p-1}\phi ^{2p}+c_{1}\int_{M}\left\vert \mathrm{Rm}\right\vert
^{p-1}\phi ^{2p}.
\end{equation*}%
Since, by \eqref{eq:unifeq}, we have that for all $0\leq t\leq \tau$
\begin{equation}\label{eq:volcomp}
  \mathrm{Vol}_{g(t)}\left(B_{g\left( 0\right)
}\left( x_{0},A\right)\right) \leq e^{cT}\mathrm{Vol}_{g(\tau)}\left(B_{g\left( 0\right)
}\left( x_{0},A\right)\right),
\end{equation}
from \eqref{b16} and \eqref{b5'}, we see that $U$ satisfies the differential inequality
\begin{equation*}
\frac{dU}{dt}\leq cU+cA^{-2p}e^{cT}\mathrm{Vol}_{g\left( \tau\right) }\left(
B_{g\left(0\right) }\left( x_{0},A\right) \right)
\end{equation*}%
on $[0, \tau]$. It follows that 
\begin{equation}
U\left( \tau\right) \leq e^{cT}\left( U\left( 0\right) +cA^{-2p}\mathrm{%
Vol}_{g\left( \tau\right) }\left( B_{g\left( 0\right) }\left(
x_{0},A\right) \right) \right) .  \label{b17}
\end{equation}%
However, again using Young's inequality and \eqref{eq:volcomp}, we have that
\begin{equation*}
U\left( 0\right) \leq c\int_{M}\left\vert \mathrm{Rm}\right\vert^{p}(x, 0)\phi ^{2p}+e^{cT}\mathrm{Vol}_{g\left( \tau\right) }\left( B_{g\left(
0\right) }\left( x_{0},A\right) \right) .
\end{equation*}%
Therefore, there exists a constant $c>0$ so that 
\begin{align*}
\begin{split}
&\int_{M}\left\vert \mathrm{Rm}\right\vert^{p}(x, \tau)\phi
^{2p}\\
&\qquad\leq ce^{cT}\left( \int_{M}\left\vert \mathrm{Rm}\right\vert^{p}(x, 0)\phi ^{2p}+\left( 1+A^{-2p}\right) \mathrm{Vol}_{g\left(
\tau\right) }\left( B_{g\left( 0\right) }\left( x_{0},A\right) \right)
\right),
\end{split}
\end{align*}%
for all $0\leq \tau\leq T$. This completes the proof of Proposition \ref{Int}. 
\begin{remark}\label{rem:p2}
Our application to Theorem \ref{Rm} only requires the validity of the above estimate for sufficiently large $p$,
and the proof above requires $p\geq 3$. However, an analogous estimate can be proven for $p=2$ 
by a slightly more detailed analysis of the terms in the evolution equation for $\Rm$. 
The idea is to write
\begin{equation*}
  \pdt R_{ijkl} = \nabla_jD_{kli} - \nabla_iD_{klj} - R_{ijpl}R_{pk} - R_{ijkp}R_{pl}.
\end{equation*}
where $D_{ijk} \doteqdot \operatorname{div}(\Rm)_{ijk} = \nabla_lR_{ijkl} = \nabla_iR_{jk} - \nabla_jR_{ik}$,
and use the Tachibana-type identity
\[
 \int_M\langle\nabla_j D_{kli} - \nabla_i D_{klj}, R_{ijkl}\rangle\phi^2 \leq -2\int_M|D|^2\phi^2 + 4\int_M|D||\Rm||\nabla \phi|\phi
\]
to estimate $\frac{d}{dt}\int_M|\Rm|^2\phi^2$ in place of \eqref{b3}.
\end{remark}

\section{Proof of Theorem \protect{\ref{Rm}}}
\label{sec:rmproof}

It is now straightforward to obtain a pointwise estimate from Proposition \ref{Int} using an iteration argument.

\begin{proof}[Proof of Theorem \ref{Rm}]
As before, we may assume $K=1$. Let us fix some $0\leq t\leq T$. Since $\left\vert \mathrm{Ric}\right\vert\leq 1$ on $B_{g\left( 0\right) }\left(
x_{0},A\right)\times [0,T]$, we have from Proposition \ref%
{Int} that 
\begin{align}
\begin{split}
\label{rmp}
	&\int_{B_{g\left( 0\right) }\left( x_{0},\frac{1}{2}A\right)
	}\left\vert \mathrm{Rm}\right\vert^{p}(x, t) \\
	&\phantom{\int}\leq
	ce^{cT}\int_{B_{g\left( 0\right) }\left( x_{0},A\right) }\left\vert 
	\mathrm{Rm}\right\vert^{p}(x, 0) \\
	&\phantom{\int}+c\left( 1+A^{-2p}\right) e^{cT}\ \mathrm{Vol}_{g\left( t\right) }\left(
	B_{g\left( 0\right) }\left( x_{0},A\right) \right) .
\end{split}
\end{align}

By the Bishop-Gromov volume
comparison theorem and (\ref{eq:unifeq}) we get 
\begin{equation*}
\frac{\mathrm{Vol}_{g\left( t\right) }\left( B_{g\left( 0\right) }\left(
	x_{0},A\right) \right) }{\mathrm{Vol}_{g\left( t\right) }\left(
	B_{g\left( 0\right) }\left( x_{0},\frac{1}{2}A\right) \right) }\leq
ce^{cT}\frac{\mathrm{Vol}_{g\left( 0\right) }\left( B_{g\left( 0\right) }\left(
	x_{0},A\right) \right) }{\mathrm{Vol}_{g\left( 0\right) }\left(
	B_{g\left( 0\right) }\left( x_{0},\frac{1}{2}A\right) \right) }\leq
ce^{c(T+A)}.
\end{equation*}
Consequently, it follows from (\ref{eq:volcomp}) and  (\ref{rmp}) that for any $p\geq 3,$%
\begin{equation}
 \left( \fint_{B_{g\left( 0\right) }\left( x_{0},\frac{1}{2}A\right)
}\left\vert \mathrm{Rm}\right\vert^{p}(x, t)\right) ^{\frac{1%
}{p}} \leq ce^{cT+cA}\left( \Lambda _{0}+\left( 1+A^{-2}\right) \right),
\label{d2}
\end{equation}%
where 
\begin{equation*}
\Lambda _{0}:=\sup_{B_{g\left( 0\right) }\left( x_{0},A\right) }\left\vert 
\mathrm{Rm}\right\vert(x, 0).
\end{equation*}%

Now, by \eqref{eq:rmnormev} of Lemma \ref{L}, we have 
\begin{equation}
\left( \Delta -\partial _{t}\right) \left\vert \mathrm{Rm}\right\vert \geq
-\left\vert \mathrm{Rm}\right\vert ^{2}.  \label{d3}
\end{equation}
We proceed to apply De Giorgi-Nash-Moser iteration to (\ref{d3}) using (\ref{d2}) for 
$p>n,$ depending only on the dimension $n$. Note that the Ricci curvature bound and (\ref{eq:unifeq}) imply a uniform bound on the Sobolev constant of $(M,g(t))$ on $
B_{g\left( 0\right) }\left( x_{0},\frac{1}{2}A\right)$. From the argument in \cite%
{Li}, Ch. 19, slightly modified as the metrics evolve in $t$, see e.g. \cite{Ye}, we then obtain 
\begin{equation*}
\left\vert \mathrm{Rm}\right\vert(x_0, T)\leq ce^{cT+cA}\left(
1+\left( T^{-1}+A^{-2}\right) ^{\beta }+\Lambda _{0}^{\alpha }\right) \left(
\Lambda _{0}+\left( 1+A^{-2}\right) \right) ,
\end{equation*}%
for some $\alpha ,\beta $ depending only on $n$. This proves the theorem. 
\end{proof}

\section{On the rate of blow-up of Ricci curvature}
\label{sec:blowup}

Theorem \ref{Rm} implies that if $g(0)$ is complete, with bounded curvature, and the Ricci flow exists on $[0,T)$ but cannot be extended past time $T < \infty$, 
then the Ricci curvature $\mathrm{Ric}$ cannot be uniformly bounded on $[0,T)$. The following theorem strengthens this conclusion.
\begin{theorem}
	\label{Blowup}Let $\left( M^{n},g\left( t\right) \right)$ be a smooth
	solution of the Ricci flow defined on $[0,T)$, so that each $(M,g(t))$ is complete and has bounded curvature. Assume that the Ricci flow cannot be extended past time $T$. Then
	there exists a constant $\varepsilon$, depending only on $n$, and a
	sequence $t_{i}\nearrow T$ so that 
	\begin{equation*}
		\sup_{M}\left\vert \mathrm{Ric}\right\vert \left( x,t_{i}\right) \geq \frac{%
			\varepsilon }{2(T-t_{i})}. 
	\end{equation*}
\end{theorem}

\begin{proof}[Proof of Theorem \ref{Blowup}] 

We argue by contradiction. 
Fix some $t_{0}\in \left[ 0,T\right) $ and assume that 
	\begin{equation}
		\sup_{M}\left\vert \mathrm{Ric}\right\vert \left( x,t\right) <
		\frac{\varepsilon }{2(T-t)},  \label{m1}
	\end{equation}%
	for all $t\in [t_{0},T)$.	
	We will show that this leads to a contradiction if $\varepsilon = \varepsilon(n)$ is chosen
	sufficiently small.
By hypothesis we know that there
	exists a constant $C$ so that 
	\begin{equation}
		\sup_{M}\left\vert \mathrm{Rm}\right\vert
		\left( x,t_0\right) \leq C.  \label{m2}
	\end{equation}%
	Throughout the proof, we will denote by $c$ a constant depending only on the
	dimension $n$, and by $C$ a constant depending only on $T,$ $t_{0}$ and $%
	\sup_{M}\left\vert \mathrm{Rm}\right\vert
	\left( x,t_0\right) $.  We will also take $p = p(n)$ to be a suitably large constant
	to be specified below in terms of the dimension. Ultimately, we will first specify the value of $p(n)$
	and then set the value of $\varepsilon(n) = \varepsilon(n, p)$ accordingly.

Fix some $x_{0}\in M$ and choose the cut-off function
\begin{equation}
	\label{phi_1}
	\phi \left( x\right) :=\left( 1-d_{g\left(t_0\right) }\left( x_{0},x\right)\right)_{+},
\end{equation}	
with support $\overline{B_{g\left( t_0\right) }\left( x_{0},1\right)}$.
	 Now fix $\tau \in \left[ t_{0},T\right)$.  
	 Applying
	Proposition \ref{A} on $\left[ t_{0},\tau \right]$, with $\Omega = B_{g(t_0)}(x_0, 2)$, it follows that there exist 
	$c_{1}= c_1(n, p)$ and $c_{2} = c_2(n, p)$, so that 
	\begin{align}
	\begin{split}
	\label{m3}
	&\frac{d}{dt}\left( \int_{M}\left\vert \mathrm{Rm}\right\vert ^{p}\phi
	^{2p}+\frac{1}{K}\int_{M}\left\vert \mathrm{Ric}\right\vert ^{2}\left\vert 
	\mathrm{Rm}\right\vert ^{p-1}\phi ^{2p}+c_{1}K\int_{M}\left\vert \mathrm{Rm}%
	\right\vert ^{p-1}\phi ^{2p}\right)    \\
	&\phantom{\frac{d}{dt}}\leq c_{2}K\int_{M}\left\vert \mathrm{Rm}\right\vert ^{p}\phi
	^{2p}+c_{2}K\int_{M}\left\vert \nabla \phi \right\vert ^{2p}   \\
	&\phantom{\frac{d}{dt}}-\frac{K^{\prime }}{K^{2}}\int_{M}\left\vert \mathrm{Ric}\right\vert
	^{2}\left\vert \mathrm{Rm}\right\vert ^{p-1}\phi ^{2p}+c_{1}K^{\prime}\int_{M}\left\vert \mathrm{Rm}\right\vert ^{p-1}\phi ^{2p},  
	\end{split}
	\end{align}
	for all $t\in \left[ t_{0},\tau \right]$. We choose 
	\begin{equation}
		K\left( t\right) :=\frac{\varepsilon }{T-t},  \label{m4}
	\end{equation}%
	and we assume that $\varepsilon =\varepsilon \left(n, p\right) $ is small
	enough so that $c_{2}\varepsilon <1$. 
	
	It follows from (\ref{m3}) and (\ref{m4}) that the function 
	\begin{equation*}
		U\left( t\right) :=\int_{M}\left\vert \mathrm{Rm}\right\vert ^{p}\phi ^{2p}+%
		\frac{1}{K}\int_{M}\left\vert \mathrm{Ric}\right\vert ^{2}\left\vert \mathrm{%
			Rm}\right\vert ^{p-1}\phi ^{2p}+c_{1}K\int_{M}\left\vert \mathrm{Rm}%
		\right\vert ^{p-1}\phi ^{2p}
	\end{equation*}%
	satisfies 
	\begin{equation}
		\frac{dU}{dt}\leq \frac{1}{T-t}U+cK\int_{M}\left\vert \nabla \phi
		\right\vert ^{2p}.  \label{m5}
	\end{equation}%
	Let us note that (\ref{m1}) and (\ref{m2}) imply%
	\begin{equation}
		\frac{1}{C}\left( T-t\right) ^{\varepsilon }g_{ij}\left( t_0\right) \leq
		g_{ij}\left( t\right) \leq \frac{C}{\left( T-t\right) ^{\varepsilon }}%
		g_{ij}\left( t_0\right) .  \label{m6}
	\end{equation}%
	By (\ref{m6}) and the definition of $\phi$ in (\ref{phi_1}), we then have that%
	\begin{equation*}
		\left\vert \nabla \phi \right\vert^2 \leq \frac{C}{\left( T-t\right)
			^{\varepsilon }}
	\end{equation*}%
	and 
	\begin{equation*}
		\mathrm{Vol}_{g\left( t\right) }\left( B_{g\left(t_0\right) }\left(
		x_{0},1\right) \right) \leq \frac{C}{\left( T-t\right) ^{\frac{n\varepsilon}{2}}}%
		\mathrm{Vol}_{g\left( t_0\right) }\left( B_{g\left( t_0\right) }\left(
		x_{0},1\right) \right) .
	\end{equation*}%
	Consequently, we deduce from (\ref{m5}) that 
	\begin{align}
	\begin{split}\label{m7}
		\frac{dU}{dt}&\leq \frac{1}{T-t}U+\frac{C}{\left( T-t\right) ^{\left(
				p+\frac{n}{2}\right) \varepsilon +1}}\mathrm{Vol}_{g\left( t_0\right) }\left(
		B_{g\left( t_0\right) }\left( x_{0},1\right) \right) 
	  \end{split}
		\end{align}%
	for all $t\in \left[ t_{0},\tau \right] $. Integrating (\ref{m7}) from $%
	t=t_{0}$ to $t=\tau$, it follows that 
	\begin{equation*}
		U\left( \tau \right) \leq \frac{C}{T-\tau}\mathrm{Vol}_{g\left(t_0\right)
		}\left( B_{g\left( t_0\right) }\left( x_{0},1\right) \right) ,
 \end{equation*}%
	by choosing $\varepsilon =\varepsilon \left(n, p\right) $ small enough in (\ref%
	{m7}). Using the Bishop-Gromov volume comparison theorem for the metric $g(t_0)$, this proves that 
	\begin{equation}
	\label{Lp_1}
		\left(\fint_{B_{g\left(t_0\right) }\left( x_{0},\frac{1}{2}\right) }\left\vert 
		\mathrm{Rm}\right\vert ^{p}\left( x,t\right) \right)^\frac{1}{p}\leq \frac{C}{(T-t)^\frac{1}{p}},
	\end{equation}%
	for all $t\in \left[ t_{0},T\right) $. 
	
	Now, we have a Sobolev
	inequality for $g\left(t_0\right) $ on the ball $B_{g\left(t_0\right) }\left( x_{0},%
	\frac{1}{2}\right)$ of the form 
	\begin{equation*}
		\left( \fint_{M}\phi ^{2\mu }dv_{t_0}\right) ^{\frac{1}{\mu }}\leq
		C_{S}\fint_{M}\left\vert \nabla \phi \right\vert _{g\left(t_0\right)
		}^{2}dv_{t_0},
	\end{equation*}%
	for some $\mu = \mu(n) \leq n/(n-2)$,
	valid for any $\phi$ supported in $B_{g\left(t_0\right) }\left( x_{0},\frac{1}{2%
	}\right) $. We may assume $C_S \geq 1$. Here $dv_{t}$ denotes the volume form of $M$ with
	respect to the metric $g\left( t\right) $. Since $C_{S}$ only depends on the lower bound of the Ricci
	curvature of $g\left( t_0\right) $ on $B_{g\left( t_0\right) }\left(
	x_{0},\frac{1}{2}\right) $, we have $C_{S}\leq C$ according to our convention. Using (\ref{m6}), we get a
	Sobolev inequality for $g\left( t\right) $ on $B_{g\left(t_0\right) }\left(
	x_{0},\frac{1}{2}\right)$ as well, of the form%
	\begin{equation*}
		\left( \fint_{M}\phi ^{2\mu }dv_{t}\right) ^{\frac{1}{\mu }}\leq \frac{C}{%
			\left( T-t\right) ^{c\varepsilon }}\fint_{M}\left\vert \nabla \phi
		\right\vert _{g\left( t\right) }^{2}dv_{t}.
	\end{equation*}%
	
	We now claim that there exists $c_3>0$, depending only on $n$, so that
	\begin{equation}
		\left\vert \mathrm{Rm}\right\vert \left( x_{0},t\right) \leq \frac{C}{\left(
			T-t\right) ^{c_3(\varepsilon + \frac{1}{p})}},  \label{m8}
	\end{equation}%
	for all $t\in [(T+t_0)/2,T)$. Indeed, we may argue as in the proof of Theorem \ref{Rm}, using DeGiorgi-Nash-Moser
	iteration. We will do this in somewhat greater detail than before since it is now crucial to our argument
	that we carefully track the various dependencies of the constants. We will closely follow the argument in Chapter 19 of \cite{Li}.
	
	Using that 
	\begin{equation}
	\label{eqn}
		\left( \partial _{t}-\Delta \right) u\leq cfu,
	\end{equation}%
	where $f=u=\left\vert \mathrm{Rm}\right\vert ,$ it follows that 
	\begin{equation}
		-\int_{M}u^{2a-1}\left( \Delta u\right) \varphi
		^{2}+\int_{M}u^{2a-1}u_{t}\varphi ^{2}\leq c\int_{M}u^{2a}f\varphi ^{2},
		\label{m9}
	\end{equation}%
	for any nonnegative Lipschitz function $\varphi$, compactly supported in $B_{g\left( t_0\right)
	}\left( x_{0},\frac{1}{2}\right) $. Here and below, the integration is taken with
	respect to the measure $dv_{t}$, and we assume $a\ge 1$. 
	
	Integrating by parts, it follows that 
	\begin{equation}
		-\int_{M}u^{2a-1}\left( \Delta u\right) \varphi ^{2}=2\int_{M}\left\langle
		\nabla u,\nabla \varphi \right\rangle u^{2a-1}\varphi
		+(2a-1)\int_{M}\left\vert \nabla u\right\vert ^{2}u^{2a-2}\varphi ^{2}.
		\label{m10}
	\end{equation}%
	Observe that (\ref{m9}) and (\ref{m10}) imply 
	\begin{equation*}
		\frac{1}{2}\frac{d}{dt}\left( \int_{M}u^{2a}\varphi ^{2}\right)
		+\int_{M}\left\vert \nabla \left( u^{a}\varphi \right) \right\vert ^{2}\leq
		ca\int_{M}u^{2a}f\varphi ^{2}+\int_{M}\left\vert \nabla \varphi \right\vert
		^{2}u^{2a}.
	\end{equation*}%
	For fixed $t_{0}<s<s+v<T$, we multiply the above inequality with the Lipschitz function 
	\begin{equation*}
		\psi \left( t\right) :=\left\{ 
		\begin{array}{c}
			0 \\ 
			\frac{t-s}{v} \\ 
			1%
		\end{array}%
		\right. 
		\begin{array}{l}
			\text{for }t_{0}\leq t\leq s \\ 
			\text{for }s<t\leq s+v \\ 
			\text{for }s+v<t\leq T%
		\end{array}%
	\end{equation*}%
	and obtain 
	\begin{align}
	\begin{split}
	\label{m11}
		&\frac{1}{2}\frac{d}{dt}\left( \psi ^{2}\int_{M}u^{2a}\varphi ^{2}\right)
		+\psi ^{2}\int_{M}\left\vert \nabla \left( u^{a}\varphi \right) \right\vert
		^{2} \\
		&\leq ca\psi ^{2}\int_{M}u^{2a}f\varphi ^{2}+\psi ^{2}\int_{M}\left\vert
		\nabla \varphi \right\vert ^{2}u^{2a}+\psi \psi ^{\prime
		}\int_{M}u^{2a}\varphi ^{2}.  
	\end{split}
	\end{align}%
	Integrating (\ref{m11}) from $t=t_{0}$ to $t=\tau$, where $\tau \in \left(
	t_{0},T\right) $ is fixed, then implies 
	\begin{align}
	\begin{split}
	\label{m12}
		&\frac{1}{2}\psi ^{2}\left( \tau \right) \int_{M}u^{2a}\varphi
		^{2}+\int_{t_{0}}^{\tau }\psi ^{2}\int_{M}\left\vert \nabla \left(
		u^{a}\varphi \right) \right\vert ^{2} \\
		&\leq ca\int_{t_{0}}^{\tau }\psi
		^{2}\int_{M}u^{2a}f\varphi ^{2}+\int_{t_{0}}^{\tau }\psi
		^{2}\int_{M}\left\vert \nabla \varphi \right\vert ^{2}u^{2a}  \\
		&+\int_{t_{0}}^{\tau }\psi \psi ^{\prime }\int_{M}u^{2a}\varphi ^{2}. 
	\end{split}
	\end{align}%
	
	The integrals above are taken with respect to the measure $dv_t$ associated to $g\left(
	t\right)$.  Now, using (\ref{m6}), we will rewrite the inequality (\ref{m12}) in terms of the measure $dv_{t_0}$
	and the metric $g(t_0)$. This can be done at the expense of decreasing the constant on the left hand side
	by an amount proportional to an appropriate power of $T-\tau$.  So transformed, the inequality reads
	\begin{align}
	\begin{split}
	\label{m13}
		&\frac{1}{C}\left( T-\tau\right) ^{c_{0}\varepsilon }\left( \frac{1}{2}\psi
		^{2}\left( \tau \right) \int_{M}u^{2a}\varphi ^{2}dv_{t_0}+\int_{t_{0}}^{\tau
		}\psi ^{2}\int_{M}\left\vert \nabla \left( u^{a}\varphi \right) \right\vert
		_{g\left(t_0\right) }^{2}dv_{t_0}\right)  \\
		&\quad\leq ca\int_{t_{0}}^{\tau }\psi ^{2}\int_{M}u^{2a}f\varphi ^{2}dv_{t_0}+\int_{t_{0}}^{\tau }\psi ^{2}\int_{M}\left\vert \nabla \varphi
		\right\vert _{g\left(t_0\right) }^{2}u^{2a}dv_{t_0}  \\
		&\quad\phantom{\leq}+\int_{t_{0}}^{\tau }\psi
		\psi ^{\prime }\int_{M}u^{2a}\varphi ^{2}dv_{t_0},  
		\end{split}
	\end{align}%
	for some constant $c_{0}$ depending only on $n$. 
	
	Since \eqref{m13} holds relative to the fixed manifold $\left( M,g\left(t_0\right) \right)$,
	we can now apply the above Sobolev inequality and follow the remainder of the argument for the DeGiorgi-Nash-Moser iteration process in 
	Chapter 19 of \cite{Li} to conclude that, if $p > \mu/(\mu-1)$, we have
\begin{equation*}
\left\vert \mathrm{Rm}\right\vert \left( x_{0},t\right) \leq 
	  \frac{CQ(t)}{(T-t)^{c\varepsilon}}\left(1 + Q^{\frac{p(\mu-1)}{\mu(p-1)-p}}\left(t\right) + \frac{1}{t-t_0}\right)
	  ^{\frac{2\mu-1}{p(\mu-1)}},
\end{equation*}
for some constant $c>0$ depending only on $n$, where
\[
  Q(t) := \sup_{\tau\in [t_0, t]}\left(\fint_{B_{g\left(t_0\right)
			}\left( x_{0},\frac{1}{2}\right) }\left\vert \mathrm{Rm}\right\vert
		^{p}\left( x,\tau\right) \right)^{\frac{1}{p}}.
\]
Consequently, if $p > \mu/(\mu-1)+ 1$ and $t \geq (T+t_0)/2$, it follows that 
         \begin{align}
	\begin{split}\label{eq:moser}
	  \left\vert \mathrm{Rm}\right\vert \left( x_{0},t\right)\leq \frac{CQ(t)}{(T-t)^{c\varepsilon}}\left(1 + Q^p(t)\right)
	  ^{\frac{2\mu-1}{p(\mu-1)}}.
	\end{split}
\end{align}
Hence,  using \eqref{Lp_1}, we have that
	\begin{align*}
		\left\vert \mathrm{Rm}\right\vert \left( x_{0},t\right)  
		&\leq \frac{C}{\left( T-t\right) ^{c_3(\varepsilon +\frac{1}{p})}}
	\end{align*}%
	for some $c_3 = c_3(n)$ and all $t\in [(T+t_0)/2, T)$ as claimed.
	
	Thus, if we now choose, say, $p = p(n) > \max\{\mu/(\mu-1) + 1, 4c_3\}$
	and then choose $\varepsilon = \varepsilon(n, p)$ so that $\varepsilon < 1/(4c_3)$ (in addition to the restrictions we have already imposed
	above), it follows that we have
	\begin{equation}
	\sup_{M}\left\vert \mathrm{Rm}\right\vert \left( x,t\right) \leq \frac{C}{%
			\left( T-t\right) ^{\frac{1}{2}}}  \label{m14}
	\end{equation}
	for all $t\in [(T+t_0)/2, T)$.

	On the other hand, each $(M,g(t))$ is complete and has bounded curvature, so we may apply the parabolic maximum principle
	to the equation
	\[
	  \partial_t|\Rm|^2 \leq \Delta |\Rm|^2 + 16|\Rm|^3
	\]
	and use that $T$ is a singular time for the solution
	to deduce that 
		\begin{equation*}
		\sup_{M}\left\vert \mathrm{Rm}\right\vert \left( x,t\right) \geq \frac{1}{%
			8\left( T-t\right) },
	\end{equation*}%
	as in Lemma 8.7 of \cite{ChowLuNi}.  This contradicts (\ref{m14}) and proves the theorem.
\end{proof}

\end{document}